\documentclass{amsart}%
\usepackage{amsfonts}
\usepackage{amsmath}
\usepackage{amssymb}
\usepackage{graphicx}%
\providecommand{\U}[1]{\protect\rule{.1in}{.1in}}
\newtheorem{theorem}{Theorem}
\theoremstyle{plain}

\newtheorem{corollary}{Corollary}

\newtheorem{lemma}{Lemma}

\newtheorem{proposition}{Proposition}
\newtheorem{remark}{Remark}

\numberwithin{equation}{section}
\begin{document}
\title[The Square of the Legendre Polynomials Operator]{Glazman-Krein-Naimark Theory, Left-Definite Theory and the Square of the
Legendre Polynomials Differential Operator }
\author{Lance L. Littlejohn}
\address{Department of Mathematics, Baylor University, One Bear Place \#97328, Waco, TX 76798-7328}
\email{Lance\_Littlejohn@baylor.edu}
\author{Quinn Wicks}
\address{Department of Mathematics, Baylor University, One Bear Place \#97328, Waco, TX 76798-7328}
\email{Quinn\_Wicks@baylor.edu}
\date{September 26, 2015 (Square of Legendre Operator-LW. tex)}
\dedicatory{We dedicate this paper to the memory of W. N. (Norrie) Everitt (1924-2011).}
\begin{abstract}
As an application of a general left-definite spectral theory, Everitt,
Littlejohn and Wellman, in 2002, developed the left-definite theory associated
with the classical Legendre self-adjoint second-order differential operator
$A$ in $L^{2}(-1,1)$ which has the Legendre polynomials $\{P_{n}%
\}_{n=0}^{\infty}$ as eigenfunctions. As a consequence, they explicitly
determined the domain $\mathcal{D}(A^{2})$ of the self-adjoint operator
$A^{2}.$ However, this domain, in their characterization, does not contain
boundary conditions. In fact, this is a general feature of the left-definite
approach developed by Littlejohn and Wellman. Yet, the square of the
second-order Legendre expression is in the limit-4 case at each end point
$x=\pm1$ in $L^{2}(-1,1)$ so $\mathcal{D}(A^{2})$ should exhibit four boundary
conditions. In this paper, we show that this domain can, in fact, be expressed
using four separated boundary conditions using the classical GKN
(Glazman-Krein-Naimark) theory. In addition, we determine a new
characterization of $\mathcal{D}(A^{2})$ that involves four \textit{non-GKN}
boundary conditions. These new boundary conditions are surprisingly simple -
and natural - and are equivalent to the boundary conditions obtained from the
GKN theory.

\end{abstract}
\maketitle

\section{Introduction}

The analytical study of the classical second-order Legendre differential
expression
\[
\ell\lbrack y](x)=-\left(  (1-x^{2})y^{\prime}(x)\right)  ^{\prime}%
\]
has a long and rich history stretching back to the seminal work of H. Weyl in
1910 \cite{Weyl} and E. C. Titchmarsh in 1940 \cite{Titchmarsh}. Part, if not
most, of the reason for the importance of this second-order expression lies in
the fact that the Legendre polynomials $\{P_{n}\}_{n=0}^{\infty}$ are
solutions. More specifically, the Legendre polynomial $y=P_{n}(x),$ for
$n\in\mathbb{N}_{0},$ is a solution of the eigenvalue equation%
\[
\ell\lbrack y](x)=n(n+1)y(x).
\]
In the Hilbert space $L^{2}(-1,1),$ there is a continuum of self-adjoint
operators generated by $\ell\lbrack\cdot].$ One such operator $A$ stands out
from the rest: this is the Legendre polynomials operator, so named because the
Legendre polynomials $\{P_{n}\}_{n=0}^{\infty}$ are eigenfunctions of $A.$ We
review properties of this operator in Section \ref{Legendre SA Operator}. \ 

In the mid 1970's, \AA . Pleijel wrote two papers (see \cite{Pleijel1} and
\cite{Pleijel}) on the Legendre expression from a left-definite spectral point
of view. W. N. Everitt's contribution \cite{Everitt-1980} continued this
left-definite study in addition to detailing an in-depth analysis of the
Legendre expression in the right-definite setting $L^{2}(-1,1)$ where he
discovered new properties of functions in the domain $\mathcal{D}(A)$ of $A.$
In \cite{Krall-Littlejohn}, A. M.\ Krall and Littlejohn considered properties
of the Legendre expression under the left-definite energy norm. In 2000, R.
Vonhoff extended Everitt's results in \cite{Vonhoff} with an extensive study
of $\ell\lbrack\cdot]$ in its (first) left-definite setting. In 2002, Everitt,
Littlejohn and Mari\'{c} \cite{Everitt-Littlejohn-Maric} published further
results in which they gave several equivalent conditions for functions to
belong to $\mathcal{D}(A);$ this result is given below in Theorem \ref{ELM}.
We also refer the reader to the paper \cite{Littlejohn and Zettl} by
Littlejohn and Zettl where the authors determine all self-adjoint operators,
generated by the Legendre expression $\ell\lbrack\cdot],$ in the Hilbert
spaces $L^{2}(-1,1),$ $L^{2}(-\infty,-1),$ $L^{2}(1,\infty)$ and
$L^{2}(\mathbb{R}).$

Littlejohn and Wellman \cite{Littlejohn and Wellman}, in 2002, developed a
general left-definite theory for an unbounded self-adjoint operator $T$
bounded below by a positive constant in a Hilbert space $H=(V,(\cdot,\cdot)),$
where $V$ denotes the underlying (algebraic) vector space and $H$ is the
resulting topological space induced by the norm $\left\Vert \cdot\right\Vert $
and inner product $(\cdot,\cdot).$ In a nutshell, the authors construct a
continuum of Hilbert spaces $\{H_{r}=(V_{r},(\cdot,\cdot)_{r})\}_{r>0}$,
forming a Hilbert scale, generated by positive powers of $T$. The authors
called these Hilbert spaces \textit{left-definite spaces}; they are
constructed using the Hilbert space spectral theorem (see \cite{Rudin}) for
self-adjoint operators.

It is a difficult problem, in general, to explicitly determine the domain of a
power of an unbounded operator. However, the authors in \cite{Littlejohn and
Wellman} prove that $V_{r}=\mathcal{D}(T^{r/2})$ and $(f,g)_{r}=(T^{r/2}%
f,T^{r/2}g).$ Furthermore, in many practical applications, as the authors
demonstrate in \cite{Littlejohn and Wellman}, the computation of the vector
spaces $V_{r}$ and inner products $(\cdot,\cdot)_{r}$ is surprisingly not
difficult. In a subsequent paper, Everitt, Littlejohn and Wellman
\cite{Everitt-Littlejohn-Wellman} applied this theory to the Legendre
polynomials operator $A$. Among other results, the authors explicitly compute
the domains of $\mathcal{D}(A^{n/2})$ for each $n\in\mathbb{N}$. Specifically,
they proved
\begin{equation}
\mathcal{D}(A^{n/2})=\{f:(-1,1)\rightarrow\mathbb{C\mid}f,f^{\prime}%
,\ldots,f^{(n-1)}\in AC_{\mathrm{loc}}(-1,1);(1-x^{2})^{n/2}f^{(n)}\in
L^{2}(-1,1)\}\quad(n\in\mathbb{N}). \label{LD Representation General}%
\end{equation}
In particular, we see that $\mathcal{D}(A^{2})$ is explicitly given by%
\begin{equation}
B=\{f:(-1,1)\rightarrow\mathbb{C}\mid f,f^{\prime},f^{\prime\prime}%
,f^{\prime\prime\prime}\in AC_{\mathrm{loc}}(-1,1);(1-x^{2})^{2}f^{(4)}\in
L^{2}(-1,1)\}; \label{LD Representation}%
\end{equation}
the reason for using the notation $B,$ instead of $\mathcal{D}(A^{2})$, will
be made clear shortly. Of course, for $f\in B,$ we have $A^{2}f=\ell^{2}[f], $
where $\ell^{2}[\cdot]$ is the square of the Legendre differential expression
given by%
\begin{equation}
\ell^{2}[y](x)=\left(  (1-x^{2})^{2}y^{\prime\prime}(x)\right)  ^{\prime
\prime}-2\left(  (1-x^{2})y^{\prime}(x)\right)  ^{\prime}.
\label{Legendre squared expression}%
\end{equation}
Notice that, curiously, there are no `boundary conditions' given in
(\ref{LD Representation}). From the Glazman-Krein-Naimark (GKN) theory
\cite[Theorem 4, Section 18.1]{Naimark}, there should be \textit{four} such
boundary conditions. This begs an obvious question: how can we `extract'
boundary conditions from the representation of $\mathcal{D}(A^{2})$ in
(\ref{LD Representation})? In this paper, we will answer this question. It is
interesting that the condition $(1-x^{2})^{2}f^{(4)}\in L^{2}(-1,1)$ seems to
`encode' these boundary conditions. In fact, along the way, we will
characterize $\mathcal{D}(A^{2})$ in four different ways. Of course, we have
the algebraic definition
\begin{equation}
\mathcal{D}(A^{2}):=\{f\in\mathcal{D}(A)\mid Af\in\mathcal{D}(A)\}
\label{Algebraic Definition of Squared Domain}%
\end{equation}
(we will show that $\mathcal{D}(A^{2}),$ given in
(\ref{Algebraic Definition of Squared Domain}), is equal to $B,$ defined in
(\ref{LD Representation})). We will also prove that $\mathcal{D}(A^{2})$ is
characterized by GKN boundary conditions associated with a self-adjoint
operator $S$, generated by $\ell^{2}[\cdot],$ in $L^{2}(-1,1)$. Specifically,
we prove that $\mathcal{D}(A^{2})$ is equal to
\begin{align}
\mathcal{D}(S):=\{f:(-1,1)\rightarrow\mathbb{C}\mid &  f,f^{\prime}%
,f^{\prime\prime},f^{\prime\prime\prime}\in AC_{\mathrm{loc}}(-1,1);f,\ell
^{2}[f]\in L^{2}(-1,1);\label{D(S)}\\
&  \lim_{x\rightarrow\pm1}[f,1]_{2}(x)=0;\text{ }\lim_{x\rightarrow\pm
1}[f,x]_{2}(x)=0\},\nonumber
\end{align}
where $[\cdot,\cdot]_{2}$ is the sesquilinear form associated with Green's
formula and $\ell^{2}[\cdot]$ in $L^{2}(-1,1);$ this form will be defined in
Section \ref{GKN Operator}. In this paper, we also show that $\mathcal{D}%
(A^{2})$ is equal to%

\begin{align}
D:=\{f:(-1,1)\rightarrow\mathbb{C}\mid &  f,f^{\prime},f^{\prime\prime
},f^{\prime\prime\prime}\in AC_{\mathrm{loc}}(-1,1);f,\ell^{2}[f]\in
L^{2}(-1,1);\label{New Domain D}\\
&  \lim_{x\rightarrow\pm1}(1-x^{2})f^{\prime}(x)=0;\text{ }\lim_{x\rightarrow
\pm1}\left(  (1-x^{2})^{2}f^{\prime\prime}(x)\right)  ^{\prime}=0\}.\nonumber
\end{align}
This characterization of $\mathcal{D}(A^{2})$ is surprising since the boundary
conditions in (\ref{New Domain D}) are \textit{not} GKN boundary conditions;
we say that $D$ is a GKN-like domain. The boundary conditions in
(\ref{New Domain D}) are remarkably simple; indeed, they are obtained as
limits from each of the two terms in (\ref{Legendre squared expression}) minus
one derivative.

In \cite{Everitt-Chisholm-Littlejohn}, the authors first showed the smoothness
condition
\begin{equation}
f\in\mathcal{D}(A)\Rightarrow f^{\prime}\in L^{2}(-1,1).
\label{Smoothness of D(A)}%
\end{equation}
As a consequence of our results in this paper, we are able to generalize
(\ref{Smoothness of D(A)}) by proving%
\[
f\in\mathcal{D}(A^{2})\Rightarrow f^{\prime\prime}\in L^{2}(-1,1)\text{ and
}\ell\lbrack f]\in AC[-1,1];
\]
see Corollary \ref{Smoothness of D(A^2)} below.

The contents of this paper are as follows. In Section
\ref{Legendre SA Operator}, we discuss properties of the Legendre expression
and the Legendre polynomials operator $A$ in $L^{2}(-1,1).$ Section
\ref{The Square of the Legendre Operator} deals briefly with the
\textit{algebraic} definition of the square $A^{2}$ of $A.$ In Section
\ref{GKN Operator}, we define a self-adjoint operator $S$ using the GKN
Theory; this operator $S$ will ultimately be shown to be $A^{2}.$ The main
theorems proven in this paper are stated in Section \ref{Main Theorems}. A key
and indispensable analytic tool - the Chisholm-Everitt Theorem - used in the
proofs of these theorems is discussed in Section \ref{Key Integral Inequality}%
. The proof that $\mathcal{D}(A^{2})=\mathcal{D}(S)$ is given in Section
\ref{Proof of Main Theorem 1}. Section \ref{Proof of Main Theorem 2}
establishes the proof that $B=\mathcal{D}(S)$. In Section
\ref{Proof of Main Theorem 3}, we show that $\mathcal{D}(S)=D$. The proofs of
the theorems in these last three sections establish our main result, Theorem
\ref{Main Theorem 4}, which we state in Section \ref{Main Theorems}. Lastly,
in Section \ref{Concluding Remarks}, we conjecture a generalization of our
main results. Further details on all of the results contained in this
manuscript can be found in the Ph.D. thesis \cite{Wicks} of Quinn Wicks.

One final remark: to summarize, in this paper we show that our left-definite
characterization (\ref{LD Representation}) of $\mathcal{D}(A^{2})$ can be
rewritten as a GKN domain (Theorem \ref{Main Theorem 2}) and as a GKN-like
domain (Theorem \ref{Main Theorem 3}). Presumably, techniques developed in
this paper will establish, for $n\in\mathbb{N},$ that the left-definite
characterization $\mathcal{D}(A^{n}),$ given in
(\ref{LD Representation General}), can be expressed as both a GKN domain and a
GKN-like domain. However, it is important to note - see
(\ref{LD Representation General}) - that the left-definite theory also
explicitly determines the domains $\mathcal{D}(A^{n/2})$ of $A^{n/2}$ for odd,
positive integers $n.$ The GKN theory was not built to handle these operators
or domains.

\section{The Legendre Differential Expression and the Legendre Polynomials
Self-Adjoint Operator $A$\label{Legendre SA Operator}}

The classic second-order Legendre differential expression is defined by
\begin{equation}
\ell\lbrack y](x):=-\left(  (1-x^{2})y^{\prime}(x)\right)  ^{\prime}%
\quad(\text{a.e. }x\in(-1,1)). \label{Legendre DE}%
\end{equation}

The maximal operator, associated with $\ell\lbrack\cdot]$ in $L^{2}(-1,1), $
is defined by%
\begin{align*}
T_{1,\max}f  &  =\ell\lbrack f]\\
f  &  \in\Delta_{1,\mathrm{\max}},
\end{align*}
where $\Delta_{1,\mathrm{\max}}$ is the maximal domain, defined by
\begin{equation}
\Delta_{1,\mathrm{\max}}:=\{f:(-1,1)\rightarrow\mathbb{C\mid}f,f^{\prime}\in
AC_{\mathrm{loc}}(-1,1);f,\ell\lbrack f]\in L^{2}(-1,1)\}.
\label{Delta_1, max}%
\end{equation}
The corresponding minimal operator $T_{1,\mathrm{\min}}$ is defined to be
\begin{align*}
T_{1,\min}f  &  =\ell\lbrack f]\\
f  &  \in\mathcal{D}(T_{1,\min}),
\end{align*}
where $\mathcal{D}(T_{1,\min})$ is the minimal domain given by
\[
\mathcal{D}(T_{1,\min}):=\{f\in\Delta_{1,\mathrm{\max}}\mid\left.  \lbrack
f,g]_{1}(x)\right\vert _{\alpha}^{\beta}=0\text{ for all }g\in\Delta
_{1,\mathrm{\max}}\}.
\]
We note that this operator $T_{1,\min}$ is a closed, symmetric operator.
Furthermore, $T_{1,\mathrm{\max}}$ and $T_{1,\mathrm{\min}}$ are adjoints of
each other.

Green's formula, for an arbitrary compact subinterval $[\alpha,\beta]$ of
$(-1,1)$ and $f,g\in\Delta_{1,\mathrm{\max}},$ is given by
\[
\int_{\alpha}^{\beta}\ell\lbrack f](x)\overline{g}(x)dx-\int_{\alpha}^{\beta
}f(x)\overline{\ell\lbrack g]}(x)dx=\left.  [f,g]_{1}(x)\right\vert _{\alpha
}^{\beta},
\]
where the sesquilinear form $[\cdot,\cdot]_{1}$ is defined by%
\begin{equation}
\lbrack f,g]_{1}(x):=-(1-x^{2})(f^{\prime}(x)\overline{g}(x)-f(x)\overline
{g}^{\prime}(x))\quad(f,g\in\Delta_{1,\mathrm{\max}}).
\label{Sesquilinear form for Legendre}%
\end{equation}
By definition of $\Delta_{1,\mathrm{\max}}$ and H\"{o}lder's inequality, we
see that the limits%
\[
\lim_{x\rightarrow\pm1}[f,g](x)
\]
exist and are finite for all $f,g\in\Delta_{1,\mathrm{\max}}.$

The endpoints $x=\pm1$ are both regular singular endpoints, in the sense of
Frobenius, of $\ell\lbrack\cdot]$ and it is well-known that this expression is
in the limit-circle case at each endpoint. Consequently, the deficiency index
of the minimal operator $T_{1,\mathrm{\min}}$ is $(2,2).$ This implies that
there is a continuum of self-adjoint restrictions of $T_{1,\mathrm{\max}}.$
The GKN Theorem \cite[Theorem 4, Section 18.1]{Naimark} (see also \cite[Volume
II, Chapter 8]{Akhieser and Glazman} and \cite[Chapter XIII]{Dunford and
Schwartz}) provides a `recipe' for determining each of these operators. We are
interested in that particular self-adjoint restriction $A$ which has the
Legendre polynomials $\{P_{n}\}_{n=0}^{\infty} $ as eigenfunctions.

This Legendre polynomials operator $A:\mathcal{D}(A)\subset L^{2}%
(-1,1)\rightarrow L^{2}(-1,1)$ is specifically given by%
\begin{align}
Af  &  =\ell\lbrack f]\label{The operator A}\\
f  &  \in\mathcal{D}(A),\nonumber
\end{align}
where%
\begin{equation}
\mathcal{D}(A):=\{f\in\Delta_{1,\mathrm{\max}}\mid\lim_{x\rightarrow\pm
1}(1-x^{2})f^{\prime}(x)=0\}. \label{D(A)}%
\end{equation}
We note that the boundary conditions expressed in (\ref{D(A)}) are equivalent
to
\[
\lbrack f,1]_{1}(\pm1)=0\quad(f\in\Delta_{1,\mathrm{\max}}).
\]
Furthermore, it is well known that the Legendre polynomials $\{P_{n}%
\}_{n=0}^{\infty}$ form a complete (orthogonal) set of eigenfunctions of $A$
and the spectrum $\sigma(A)$ is discrete and given explicitly by%
\[
\sigma(A):=\{n(n+1)\mid n\in\mathbb{N}_{0}\}.
\]
For our purposes, it is the case that%
\begin{equation}
(Af,f)=\int_{-1}^{1}\ell\lbrack f](x)\overline{f}(x)dx=\int_{-1}^{1}%
(1-x^{2})\left\vert f^{\prime}(x)\right\vert ^{2}dx\geq0\quad(f\in
\mathcal{D}(A));\nonumber
\end{equation}
that is to say, $A$ is a positive operator. The positivity of $A$ implies that
the left-definite theory developed by Littlejohn and Wellman in
\cite{Littlejohn and Wellman} can be used to\ determine $\mathcal{D}(A^{n})$
for each $n\in\mathbb{N};$ indeed, see (\ref{LD Representation General}).

The following theorem, shown by Everitt, Littlejohn and Mari\'{c} in
\cite{Everitt-Littlejohn-Maric}, lists several equivalent conditions for a
function $f$ to belong to $\mathcal{D}(A).$ Note the surprising, and
remarkable, equivalence of conditions (ii) and (iii) (and (ii) and (v)) below;
parts (ii)\ and (v) will be of particular use to us in this paper.

\begin{theorem}
\label{ELM}Let $f\in\Delta_{1,\mathrm{\max}}$, where $\Delta_{1,\mathrm{\max}%
}$ is given in $($\ref{Delta_1, max}$).$ The following conditions are equivalent:

\begin{enumerate}
\item[(i)] $f\in\mathcal{D}(A);$

\item[(ii)] $f^{\prime}\in L^{2}(-1,1);$

\item[(iii)] $f^{\prime}\in L^{1}(-1,1);$

\item[(iv)] $f$ is bounded on $(-1,1);$

\item[(v)] $f\in AC[-1,1];$

\item[(vi)] $(1-x^{2})^{1/2}f^{\prime}\in L^{2}(-1,1);$

\item[(vii)] $(1-x^{2})f^{\prime\prime}\in L^{2}(-1,1).$
\end{enumerate}
\end{theorem}

\section{The Square of the Legendre Polynomials Operator
\label{The Square of the Legendre Operator}}

The square $A^{2}:\mathcal{D}(A^{2})\subset L^{2}(-1,1)\rightarrow
L^{2}(-1,1)$ of the Legendre polynomials operator $A$ in $L^{2}(-1,1)$ is
\textit{algebraically }defined by%
\begin{equation}
A^{2}f:=\ell^{2}[f] \label{A^2}%
\end{equation}
for $f\in\mathcal{D}(A^{2}),$ where $\mathcal{D}(A^{2})$ is defined in
(\ref{Algebraic Definition of Squared Domain}), and where
\begin{align}
\ell^{2}[y](x)  &  :=\left(  (1-x^{2})^{2}y^{\prime\prime}(x)\right)
^{\prime\prime}-2\left(  (1-x^{2})y^{\prime}(x)\right)  ^{\prime
}\label{Square of Legendre DE}\\
&  =(1-x^{2})^{2}y^{(4)}(x)-8x(1-x^{2})y^{\prime\prime\prime}(x)+(14x^{2}%
-6)y^{\prime\prime}(x)+4xy^{\prime}(x).\nonumber
\end{align}
By standard results from functional analysis (specifically, the Hilbert space
spectral theorem), it can be shown that $A^{2}$ is a self-adjoint operator in
$L^{2}(-1,1)$, the spectrum of $A^{2}$ is given by $\sigma(A^{2}%
)=\{n^{2}(n+1)^{2}\mid n\in\mathbb{N}_{0}\}$ and the Legendre polynomials
$\{P_{n}\}_{n=0}^{\infty}$ are eigenfunctions of $A^{2}.$

It is natural to ask whether we can explicitly describe the functions in the
domain $\mathcal{D}(A^{2})$ similar to how we characterize elements in
$\mathcal{D}(A)$ as in $($\ref{D(A)}$)$ (or by Theorem \ref{ELM}). In the next
section, we identify $A^{2}$ with a self-adjoint operator $S$ obtained through
an application of the GKN theory.

\section{A GKN Self-Adjoint Operator Generated by the Square of the Legendre
Differential Expression\label{GKN Operator}}

The maximal domain $\Delta_{2,\mathrm{\max}}$ in $L^{2}(-1,1)$ associated with
the square of the\ Legendre expression $\ell^{2}[\cdot],$ defined in
(\ref{Square of Legendre DE}), is given by
\begin{equation}
\Delta_{2,\mathrm{\max}}:=\{f:(-1,1)\rightarrow\mathbb{C\mid}f,f^{\prime
},f^{\prime\prime},f^{\prime\prime\prime}\in AC_{\mathrm{loc}}(-1,1);f,\ell
^{2}[f]\in L^{2}(-1,1)\}. \label{Maximal Domain for Square}%
\end{equation}
The sesquilinear form $[\cdot,\cdot]_{2}(\cdot):$ $\Delta_{2,\mathrm{\max}}$ x
$\Delta_{2,\mathrm{\max}}$ x $(-1,1),$ associated with $\ell^{2}[\cdot],$ is
defined by%
\begin{align}
\lbrack f,g]_{2}(x)  &  :=\left(  (1-x^{2})^{2}f^{\prime\prime}(x)\right)
^{\prime}\overline{g}(x)-\left(  (1-x^{2})^{2}\overline{g}^{\prime\prime
}(x)\right)  ^{\prime}f(x)\nonumber\\
&  -(1-x^{2})^{2}f^{\prime\prime}(x)\overline{g}^{\prime}(x)+(1-x^{2}%
)^{2}f^{\prime}(x)\overline{g}^{\prime\prime}%
(x)\label{Sesquilinear Form for Square}\\
&  -2(1-x^{2})f^{\prime}(x)\overline{g}(x)+2(1-x^{2})f(x)\overline{g}^{\prime
}(x)\quad(x\in(-1,1)).\nonumber
\end{align}
For $f,g\in\Delta_{2,\mathrm{\max}}$ and $[\alpha,\beta]\subset(-1,1),$
Green's formula for $\ell^{2}[\cdot]$ is given by%
\begin{equation}
\int_{\alpha}^{\beta}\ell^{2}[f](x)\overline{g}(x)dx-\int_{\alpha}^{\beta
}f(x)\overline{\ell^{2}[g]}(x)dx=[f,g]_{2}(x)\mid_{\alpha}^{\beta}.
\label{Green's Formula for the Legendre Square}%
\end{equation}
By definition of $\Delta_{2,\mathrm{\max}}$ and H\"{o}lder's inequality, we
see that the limits
\[
\lbrack f,g]_{2}(\pm1):=\lim_{x\rightarrow\pm1}[f,g]_{2}(x)
\]
exist and are finite for all $f,g\in\Delta_{2,\mathrm{\max}}.$ Clearly%
\begin{equation}
P_{n}\in\Delta_{2,\mathrm{\max}}\quad(n\in\mathbb{N}_{0}),
\label{P_n in Maximal Domain for Square}%
\end{equation}
where $P_{n}(x)$ is the $n^{th}$ degree Legendre polynomial. In particular,
the functions $1$ and $x$ \ belong to $\Delta_{2,\mathrm{\max}}.$

The endpoints $x=\pm1$ are both regular singular points, in the sense of
Frobenius, of $\ell^{2}[\cdot].$ The Frobenius indicial equation, at either
endpoint, is given by%
\[
r^{2}(r-1)^{2}=0.
\]
It follows, from the general Weyl theory, that each endpoint is in the limit-4
case so the deficiency index of the minimal operator $T_{2,\min},$ generated
by $\ell^{2}[\cdot],$ in $L^{2}(-1,1)$ is $(4,4).$ Consequently, each
self-adjoint operator, generated by $\ell^{2}[\cdot],$ in $L^{2}(-1,1) $ is
determined by restricting $\Delta_{2,\mathrm{\max}}$ to four boundary
conditions of the form
\begin{equation}
\lbrack f,f_{j}]_{2}(1)-[f,f_{j}]_{2}(-1)=0, \label{Form of GKN BCs}%
\end{equation}
where $[\cdot,\cdot]_{2}$ is given in (\ref{Sesquilinear Form for Square}) and
where $\{f_{1},f_{2},f_{3},f_{4}\}\subset\Delta_{2,\mathrm{\max}}$ is linearly
independent modulo the minimal domain $\Delta_{2,\mathrm{\min}}$ defined by%
\[
\Delta_{2,\mathrm{\min}}:=\{f\in\Delta_{2,\mathrm{\max}}\mid\left.  \lbrack
f,g]_{2}\right\vert _{-1}^{1}=0\text{ for all }g\in\Delta_{2,\mathrm{\max}%
}\}.
\]

We now identify a particular self-adjoint operator restriction $S$ of
$T_{\mathrm{\max}}$, generated by $\ell^{2}[\cdot],$ having the Legendre
polynomials $\{P_{n}\}_{n=0}^{\infty}$ as a complete set of eigenfunctions.

For $j=1,2,3,4,$ define $f_{j}\in\Delta_{2,\mathrm{\max}}\cap C^{4}[-1,1]$ by%
\begin{equation}%
\begin{array}
[c]{ll}%
f_{1}(x)=\left\{
\begin{array}
[c]{ll}%
1 & \text{near }x=1\\
0 & \text{near }x=-1,
\end{array}
\right.  \bigskip & f_{2}(x)=\left\{
\begin{array}
[c]{ll}%
0 & \text{near }x=1\\
1 & \text{near }x=-1,
\end{array}
\right. \\
f_{3}(x)=\left\{
\begin{array}
[c]{ll}%
x & \text{near }x=1\\
0 & \text{near }x=-1,
\end{array}
\right.  & f_{4}(x)=\left\{
\begin{array}
[c]{ll}%
0 & \text{near }x=1\\
x & \text{near }x=-1,
\end{array}
\right.
\end{array}
\label{BC functions}%
\end{equation}

\begin{proposition}
\label{Proposition about LI modulo Minimal Domain of Square} The functions
$\{f_{j}\}_{j=1}^{4},$ defined in $($\ref{BC functions}$),$ are linearly
independent modulo $\Delta_{2,\mathrm{\min}}.$
\end{proposition}

\begin{proof}
Calculations show that the functions $\ln(1\pm x)$ and $(1\pm x)\ln(1\pm x)$
belong to $\Delta_{2,\mathrm{\max}}.$ We modify these functions by defining
the four functions $g_{j}\in\Delta_{2,\mathrm{\max}}\cap C^{4}(-1,1)$
$(j=1,2,3,4)$
\[%
\begin{array}
[c]{ll}%
g_{1}(x)=\left\{
\begin{array}
[c]{ll}%
\ln(1-x) & \text{near }x=1\\
0 & \text{near }x=-1,
\end{array}
\right.  & g_{2}(x)=\left\{
\begin{array}
[c]{ll}%
0 & \text{near }x=1\\
\ln(1+x) & \text{near }x=-1,
\end{array}
\right.  \bigskip\\
g_{3}(x)=\left\{
\begin{array}
[c]{ll}%
(1-x)\ln(1-x) & \text{near }x=1\\
0 & \text{near }x=-1,
\end{array}
\right.  & g_{4}(x)=\left\{
\begin{array}
[c]{ll}%
0 & \text{near }x=1\\
(1+x)\ln(1+x) & \text{near }x=-1.
\end{array}
\right.
\end{array}
\]
Suppose that
\[
\sum_{j=1}^{4}\alpha_{j}f_{j}\in\Delta_{2,\mathrm{\min}};
\]
then, by definition of $\Delta_{2,\mathrm{\min}},$ we see that%
\begin{equation}
\left.  \left[  \sum_{j=1}^{4}\alpha_{j}f_{j},g\right]  _{2}\right\vert
_{-1}^{1}=0\text{\qquad}(g\in\Delta_{2,\mathrm{\max}}),
\label{Minimal Domain Condition}%
\end{equation}
where $[\cdot,\cdot]_{2}$ is the sesquilinear form defined in $($%
\ref{Sesquilinear Form for Square}$)$. A calculation shows that
\[
0=\left.  \left[  \sum_{j=1}^{4}\alpha_{j}f_{j},g_{1}\right]  _{2}\right\vert
_{-1}^{1}=-4\alpha_{3}%
\]
so $\alpha_{3}=0.$ Similarly, we find that $\alpha_{1}=\alpha_{2}=\alpha
_{4}=0$ after substituting $g=g_{2,}$ $g_{3},$ $g_{4}$ into $($%
\ref{Minimal Domain Condition}$).$ This completes the proof.
\end{proof}

It is clear that the boundary conditions%
\[
\lbrack f,f_{1}]_{2}(1)=[f,f_{3}]_{2}(1)=[f,f_{2}]_{2}(-1)=[f,f_{4}%
]_{2}(-1)=0
\]
are equivalent to the boundary conditions%
\[
\lbrack f,1]_{2}(\pm1)=[f,x]_{2}(\pm1)=0.
\]

We are now in position to define the operator $S$ which we show later (see
Section \ref{Proof of Main Theorem 1}) to be equal to the operator $A^{2},$
given in (\ref{A^2}) and (\ref{Algebraic Definition of Squared Domain}).
Indeed, let $S:\mathcal{D}(S)\subset L^{2}(-1,1)\rightarrow L^{2}(-1,1)$ be
defined by
\begin{align}
Sf  &  =\ell^{2}[f]:=\ell\lbrack\ell\lbrack f]]\label{The Operator S}\\
f  &  \in\mathcal{D}(S),\nonumber
\end{align}
where the domain $\mathcal{D}(S)$ of $S$ is defined in (\ref{D(S)}). By the
GKN Theorem \cite[Theorem 4, Section 18.1]{Naimark}, $S$ is self-adjoint in
$L^{2}(-1,1).$ Moreover, notice that for $f\in\Delta_{2,\mathrm{\max}},$
\begin{equation}
\lbrack f,1]_{2}(x)=\left(  (1-x^{2})^{2}f^{\prime\prime}(x)\right)  ^{\prime
}-2(1-x^{2})f^{\prime}(x) \label{[f,1]_2}%
\end{equation}
and%
\begin{align}
\lbrack f,x]_{2}(x)  &  =\left(  (1-x^{2})^{2}f^{\prime\prime}(x)\right)
^{\prime}x-(1-x^{2})^{2}f^{\prime\prime}(x)-2x(1-x^{2})f^{\prime}%
(x)+2(1-x^{2})f(x)\label{[f,x]_2}\\
&  =x[f,1]_{2}(x)-(1-x^{2})^{2}f^{\prime\prime}(x)+2(1-x^{2})f(x).\nonumber
\end{align}
From (\ref{[f,1]_2}) and (\ref{[f,x]_2}), it is easy to see that the Legendre
polynomials $\{P_{n}\}_{n=0}^{\infty}$ satisfy%
\[
\lbrack P_{n},1]_{2}(\pm1)=[P_{n},x]_{2}(\pm1)=0.
\]
That is to say, the Legendre polynomials $\{P_{n}\}_{n=0}^{\infty}$
$\subset\mathcal{D}(S).$ Moreover $\ell^{2}[P_{n}]=\ell\lbrack\ell\lbrack
P_{n}]]=n(n+1)\ell\lbrack P_{n}]=n^{2}(n+1)^{2}P_{n}\quad(n\in\mathbb{N}%
_{0}).$

From \cite{Naimark} and standard results in spectral theory, the following
result holds.

\begin{theorem}
\label{Self-Adjointness and Spectrum of S} The operator $S,$ defined in
$($\ref{The Operator S}$)$ and $($\ref{D(S)}$),$ is an unbounded self-adjoint
operator in $L^{2}(-1,1).$ The Legendre polynomials $\{P_{n}\}_{n=0}^{\infty}$
form a complete set of $($orthogonal$)$ eigenfunctions of $S$ in
$L^{2}(-1,1).$ The spectrum $\sigma(S)$ of $S$ is discrete and given
explicitly by
\[
\sigma(S)=\{n^{2}(n+1)^{2}\mid n\in\mathbb{N}_{0}\}.
\]

\end{theorem}

\section{Statements of the Main Theorems\label{Main Theorems}}

There are four main theorems that we prove in this paper.

\begin{theorem}
\label{Main Theorem 1}Let $\mathcal{D}(A^{2})$ and $\mathcal{D}(S)$ be given,
respectively, as in $($\ref{Algebraic Definition of Squared Domain}$)$ and
$($\ref{D(S)}$). $ Then%
\[
\mathcal{D}(A^{2})=\mathcal{D}(S).
\]

\end{theorem}

\begin{proof}
see Section \ref{Proof of Main Theorem 1}.
\end{proof}

\begin{theorem}
\label{Main Theorem 2}Let $B$ and $\mathcal{D}(S)$ be given, respectively, as
in $($\ref{LD Representation}$)$ and $($\ref{D(S)}$).$ Then
\[
B=\mathcal{D}(S).
\]

\end{theorem}

\begin{proof}
see Section \ref{Proof of Main Theorem 2}.
\end{proof}

\begin{theorem}
\label{Main Theorem 3}Let $\mathcal{D}(S)$ and $D$ be given, respectively, as
in $($\ref{D(S)}$)$ and $($\ref{New Domain D}$).$ Then%
\[
D=\mathcal{D}(S).
\]

\end{theorem}

\begin{proof}
see Section \ref{Proof of Main Theorem 3}.
\end{proof}

From these three theorems, we obtain our main result, namely

\begin{theorem}
\label{Main Theorem 4}Let $\Delta_{2,\mathrm{\max}}$, given in $($%
\ref{Maximal Domain for Square}$),$ be the maximal domain of the formal square
$\ell^{2}[\cdot]$ of the Legendre differential expression defined by
\[
\ell^{2}[y](x)=\left(  (1-x^{2})^{2}y^{\prime\prime}(x)\right)  ^{\prime
\prime}-2\left(  (1-x^{2})y^{\prime}(x)\right)  ^{\prime}\quad(x\in(-1,1))
\]
and let $[\cdot,\cdot]_{2}$ be the associated sequilinear form for $\ell
^{2}[\cdot]$ given in $($\ref{Sesquilinear Form for Square}$).$ Define the
operator $T:\mathcal{D}(T)\subset L^{2}(-1,1)\rightarrow L^{2}(-1,1)$ by
\begin{align*}
(Tf)(x)  &  =\ell^{2}[f](x)\quad(\text{a.e. }x\in(-1,1))\\
f  &  \in\mathcal{D}(T):=\mathcal{D}(A^{2}),
\end{align*}
where $\mathcal{D}(A^{2}),$ algebraically defined in $($%
\ref{Algebraic Definition of Squared Domain}$),$ is the domain of the square
of the Legendre polynomials operator $A$ defined in $($\ref{D(A)}$).$ That is
to say, $T$ is the square of the classical Legendre polynomials operator $A,$
given in $($\ref{The operator A}$)\ $and $($\ref{D(A)}$).$ Then the following
statements are equivalent:

\begin{enumerate}
\item[(i)] $f\in\mathcal{D}(T);\medskip$

\item[(ii)] $f,f^{\prime},f^{\prime\prime},f^{\prime\prime\prime}\in
AC_{\mathrm{loc}}(-1,1)$ and $(1-x^{2})^{2}f^{(4)}\in L^{2}(-1,1);\medskip$

\item[(iii)] $f\in\Delta_{2,\mathrm{\max}}$ and $[f,1]_{2}(\pm1)=[f,x]_{2}%
(\pm1)=0;\medskip$

\item[(iv)] $f\in\Delta_{2,\mathrm{\max}}$ and $\lim_{x\rightarrow\pm
1}(1-x^{2})f^{\prime}(x)=\lim_{x\rightarrow\pm1}\left(  (1-x^{2})^{2}%
f^{\prime\prime}(x)\right)  ^{\prime}=0.$
\end{enumerate}

Moreover, $T$ is a self-adjoint operator in $L^{2}(-1,1)$ having the Legendre
polynomials $\{P_{n}\}_{n=0}^{\infty}$ as a complete set of eigenfunctions in
$L^{2}(-1,1)$ and having discrete spectrum $\sigma(T^{2})$ explicitly given by%
\[
\sigma(T^{2})=\{n^{2}(n+1)^{2}\mid n\in\mathbb{N}_{0}\}.
\]

\end{theorem}

\section{A Key Integral Inequality\label{Key Integral Inequality}}

A key result in our analysis below is the following operator inequality
established by Chisholm and Everitt (CE) in \cite{Chisholm-Everitt}.

\begin{theorem}
\label{CE Theorem}$($The CE\ Theorem$)$ Let $(a,b)$ be an open interval of the
real line $($bounded or unbounded$)$ and let $w$ be a Lebesgue measurable
function that is positive $a.e.$ $x\in(a,b).$ Suppose $\varphi,\psi
:(a,b)\rightarrow\mathbb{C}$ satisfy the conditions\newline$(i)$ $\varphi
,\psi\in L_{\mathrm{loc}}^{2}((a,b);w);$\newline$(ii)$ there exists
$c\in(a,b)$ such that $\varphi\in L^{2}((a,c];w)$ and $\psi\in L^{2}%
([c,b);w);$\newline$(iii)$ for all $[\alpha,\beta]\subset(a,b)$%
\[
\int_{\alpha}^{\beta}\left\vert \varphi(x)\right\vert ^{2}w(x)dx>0\text{ and
}\int_{\alpha}^{\beta}\left\vert \psi(x)\right\vert ^{2}w(x)dx>0.
\]
Define the linear operators $A,$ $B$ $:L^{2}((a,b);w)$ $\rightarrow
L_{\mathrm{loc}}^{2}((a,b);w)$ by%
\[
(Af)(x)=\varphi(x)\int_{x}^{b}\psi(t)f(t)w(t)dt\quad(t\in(a,b);f\in
L^{2}((a,b);w)),
\]
and%
\[
(Bf)(x)=\psi(x)\int_{a}^{x}\varphi(t)f(t)w(t)dt\quad(t\in(a,b);f\in
L^{2}((a,b);w)).
\]
Let $K:(a,b)\rightarrow(0,\infty)$ be given by%
\begin{equation}
K(x):=\left(  \int_{a}^{x}\left\vert \varphi(t)\right\vert ^{2}w(t)dt\right)
^{1/2}\left(  \int_{x}^{b}\left\vert \psi(t)\right\vert ^{2}w(t)dt\right)
^{1/2}\quad(t\in(a,b)), \label{5.1}%
\end{equation}
and define $K\in\lbrack0,\infty]$ by%
\begin{equation}
K:=\sup\{K(x)\mid x\in(a,b)\}. \label{5.2}%
\end{equation}
Then a necessary and sufficient condition that $A$ and $B$ are both bounded
operators from $L^{2}((a,b);w)$ into $L^{2}((a,b);w)$ is that
\[
0<K<\infty.
\]
Moreover, the following inequalities hold%
\begin{equation}
\left\Vert Af\right\Vert \leq2K\left\Vert f\right\Vert \quad(\,f\in
L^{2}(\,(a,b);w)\,) \label{eq1.15}%
\end{equation}%
\begin{equation}
\left\Vert Bg\right\Vert \leq2K\left\Vert g\right\Vert \quad(\,g\in
L^{2}(\,(a,b);w)\,) \label{eq1.16}%
\end{equation}
where the number $K$ is defined by $($\ref{5.2}$).$ In general, the number
$2K$ appearing in both $($\ref{eq1.15}$)$ and $($\ref{eq1.16}$)$ is best
possible for these inequalities to hold.

\begin{remark}
Theorem \ref{CE Theorem}, proven by Chisholm and Everitt in 1970, was extended
in 1999 by Chisholm, Everitt and Littlejohn to the spaces $L^{p}((a,b);w)$ and
$L^{q}((a,b);w)$ where $p,q>1$ are conjugate indices; see
\cite{Everitt-Chisholm-Littlejohn}. Both Theorem \ref{CE Theorem} and its
generalization in \cite{Everitt-Chisholm-Littlejohn} have seen several
applications including a new proof of the classical Hardy integral inequality
\cite[Section 9.8, Theorem 327]{Hardy-Littlewood-Polya} $($see \cite[Example
1]{Everitt-Chisholm-Littlejohn}$)$ and numerous applications to orthogonal
polynomials $($for example, see \cite[Section 6]{Everitt-Chisholm-Littlejohn}%
$).$ Several more applications of the CE Theorem will be given in this paper.
Indeed, Theorem \ref{CE Theorem} proves to be an indispensable tool in our
analysis below.
\end{remark}
\end{theorem}

\section{Proof of Theorem \ref{Main Theorem 1}\label{Proof of Main Theorem 1}}

We now prove Theorem \ref{Main Theorem 1}, namely that $\mathcal{D}%
(A^{2})=\mathcal{D}(S),$ where $\mathcal{D}(A^{2})$ is defined in
(\ref{Algebraic Definition of Squared Domain}) and $\mathcal{D}(S)$ is given
in (\ref{D(S)}). Throughout this section, we assume that $f$ is a real-valued
function on $(-1,1).$

\begin{proof}
\underline{$\mathcal{D}(S)\subset\mathcal{D}(A^{2})$}$:$ \newline Let $f\in
D(S).$ We know that\newline

\begin{enumerate}
\item[(i)] $f,f^{\prime},f^{\prime\prime},f^{\prime\prime\prime}\in
AC_{\mathrm{loc}}(-1,1);$

\item[(ii)] $f\in L^{2}(-1,1);$

\item[(iii)] $\ell^{2}[f]\in L^{2}(-1,1)$ where $\ell^{2}[\cdot]$ is defined
by $($\ref{Square of Legendre DE}$);$

\item[(iv)] $[f,1]_{2}(\pm1)=0,$ where $[\cdot,1]_{2}(\cdot)$ is given in
$($\ref{[f,1]_2}$);$

\item[(v)] $[f,x]_{2}(\pm1)=0,$ where $[\cdot,x]_{2}(\cdot)$ is given in
$($\ref{[f,x]_2}$).$
\end{enumerate}

\noindent Taking into account the definition of $\mathcal{D}(A)$ in
$($\ref{D(A)}$)$ and $\mathcal{D(}A^{2})$ in $($%
\ref{Algebraic Definition of Squared Domain}$),$ we need to show that

\begin{enumerate}
\item[(a)] $f,f^{\prime}\in AC_{\mathrm{loc}}(-1,1);$

\item[(b)] $f\in L^{2}(-1,1);$

\item[(c)] $\ell\lbrack f]=-\left(  (1-x^{2})f^{\prime}\right)  ^{\prime
}=-(1-x^{2})f^{\prime\prime}+2xf^{\prime}\in L^{2}(-1,1);$ in fact we will
show that $\ell\lbrack f]\in AC[-1,1];$

\item[(d)] $\lim_{x\rightarrow\pm1}(1-x^{2})f^{\prime}(x)=0;$

\item[(e)] $\ell\lbrack f],$ $\ell^{\prime}[f]\in AC_{\mathrm{loc}}(-1,1); $

\item[(f)] $\ell^{2}[f]\in L^{2}(-1,1);$

\item[(g)] $\lim_{x\rightarrow\pm1}(1-x^{2})\ell^{\prime}[f]=\lim
_{x\rightarrow\pm1}(1-x^{2})\left(  (1-x^{2})f^{\prime\prime\prime
}(x)-4xf^{\prime\prime}(x)-2f^{\prime}(x)\right)  =0.$
\end{enumerate}

\noindent Clearly, $(\mathrm{a}),$ $(\mathrm{b})$ and $(\mathrm{f})$ are
satisfied. As for $(\mathrm{g}),$ note that%
\begin{align}
-(1-x^{2})\ell^{\prime}[f](x)  &  =(1-x^{2})^{2}f^{\prime\prime\prime
}(x)-4x(1-x^{2})f^{\prime\prime}(x)-2(1-x^{2})f^{\prime}(x)\nonumber\\
&  =\left(  (1-x^{2})^{2}f^{\prime\prime}(x)\right)  ^{\prime}-2(1-x^{2}%
)f^{\prime}(x)\label{A^2=S-1}\\
&  =[f,1]_{2}(x)\nonumber
\end{align}
so $(\mathrm{g})$ follows from $(\mathrm{iv})$ above. Moreover, by
$(\mathrm{i})$ and the fact that the product of a polynomial and a function
$g\in AC_{\mathrm{loc}}(-1,1)$ also belongs to $AC_{\mathrm{loc}}(-1,1),$ we
see that $(\mathrm{e})$ follows. To show $(\mathrm{c})$ note that, by
$(\mathrm{iii}),$
\begin{equation}
\ell^{2}[f](x)=\ell\lbrack\ell\lbrack f]](x)=-\left(  (1-x^{2})\ell^{\prime
}[f](x)\right)  ^{\prime}\in L^{2}(-1,1). \label{ell^2 identity}%
\end{equation}
We now apply the CE Theorem on the interval $[0,1)$ with $\psi(x)=1,\varphi
(x)=1/(1-x^{2})$ and $w(x)=1;\ $\ note that $\varphi\in L^{2}(0,1/2]$ and
$\psi\in L^{2}[1/2,1).$ A calculation shows that
\[
K^{2}(x)=\int_{0}^{x}\frac{dt}{(1-t^{2})^{2}}\cdot\int_{x}^{1}dt\quad
(x\in(0,1))
\]
is bounded on $(0,1).$ Hence we see, from Theorem \ref{CE Theorem}, that
\[
\varphi(x)\int_{x}^{1}\psi(t)\ell^{2}[f](t)w(t)dt=\dfrac{1}{1-x^{2}}\int
_{x}^{1}\ell^{2}[f](t)dt\in L^{2}[0,1).
\]
That is to say, by (\ref{ell^2 identity}),
\begin{equation}
\dfrac{1}{1-x^{2}}\left(  (1-x^{2})\ell^{\prime}[f](x)-\lim_{x\rightarrow
1}(1-x^{2})\ell^{\prime}[f](x)\right)  \in L^{2}[0,1). \label{A^2=S-2}%
\end{equation}
By $(\mathrm{iv})$ and $($\ref{A^2=S-1}$),$ we know
\[
\lim_{x\rightarrow1}(1-x^{2})\ell^{\prime}[f](x)=0.
\]
Hence, $($\ref{A^2=S-2}$)$ simplifies to
\[
\ell^{\prime}[f]\in L^{2}[0,1).
\]
A similar application of the CE Theorem on $(-1,0]$ reveals that $\ell
^{\prime}[f]\in L^{2}(-1,0]$ and thus we see that%
\[
\ell^{\prime}[f]\in L^{2}(-1,1).
\]
It follows that%
\[
\ell\lbrack f]\in AC[-1,1]\subset L^{2}(-1,1),
\]
establishing $(\mathrm{c}).$ It remains to show that $(\mathrm{d})$ holds. To
this end, observe, from $($\ref{Legendre DE}$)$ and $($%
\ref{Square of Legendre DE}$)$ that%
\[
\left(  (1-x^{2})^{2}f^{\prime\prime}(x)\right)  ^{\prime\prime}=\ell
^{2}[f](x)-2\ell\lbrack f](x).
\]
Consequently, from $(\mathrm{c})\ $and $(\mathrm{f}),$
\[
\left(  (1-x^{2})^{2}f^{\prime\prime}(x)\right)  ^{\prime\prime}\in
L^{2}(-1,1)
\]
from which we see that
\[
\left(  (1-x^{2})^{2}f^{\prime\prime}(x)\right)  ^{\prime},\text{ }%
(1-x^{2})^{2}f^{\prime\prime}(x)\in AC[-1,1].
\]
In particular, we see that the limits%
\begin{equation}
\lim_{x\rightarrow\pm1}\left(  (1-x^{2})^{2}f^{\prime\prime}(x)\right)
^{\prime} \label{Limit 2}%
\end{equation}
and
\begin{equation}
\lim_{x\rightarrow\pm1}(1-x^{2})^{2}f^{\prime\prime}(x) \label{Limit 1}%
\end{equation}
exist and are finite. Moreover, from $(\mathrm{iv})$, $(\mathrm{v})$ and
$($\ref{[f,x]_2}$),$ we see that
\begin{equation}
0=\lim_{x\rightarrow\pm1}\left(  x[f,1]_{2}(x)-[f,x]_{2}(x)\right)
=\lim_{x\rightarrow\pm1}\left(  (1-x^{2})^{2}f^{\prime\prime}(x)-2(1-x^{2}%
)f(x)\right)  . \label{Crucial Limit}%
\end{equation}
In concert with $($\ref{Limit 1}$),$ we can say that%
\[
\lim_{x\rightarrow\pm1}(1-x^{2})f(x):=r
\]
exists and is finite. We claim that $r=0;$ to show this, we deal with the
limit as $x\rightarrow1;$ a similar proof can be made as $x\rightarrow-1.$
Suppose, to the contrary, that $r\neq0;$ without loss of generality, suppose
$r>0.$ Then there exists $x^{\ast}>0$ such that%
\[
(1-x^{2})f(x)\geq\dfrac{r}{2}\text{ for }x\in\lbrack x^{\ast},1).
\]
However, in this case, we see that%
\[
\infty>\int_{-1}^{1}\left\vert f(x)\right\vert ^{2}dx\geq\int_{x^{\ast}}%
^{1}\left\vert f(x)\right\vert ^{2}dx\geq\left(  \dfrac{r}{2}\right)  ^{2}%
\int_{x^{\ast}}^{1}\dfrac{dx}{(1-x^{2})^{2}}=\infty,
\]
contradicting $(\mathrm{ii})$. Hence it follows that
\[
\lim_{x\rightarrow\pm1}(1-x^{2})f(x)=0.
\]
Consequently, we see from $($\ref{Crucial Limit}$),$ that
\[
\lim_{x\rightarrow\pm1}(1-x^{2})^{2}f^{\prime\prime}(x)=0
\]
and, hence
\begin{equation}
\lim_{x\rightarrow\pm1}(1-x)^{2}f^{\prime\prime}(x)=0. \label{Crucial Limit 2}%
\end{equation}
We are now in position to prove part $(\mathrm{d}).$ We show that
\begin{equation}
\lim_{x\rightarrow1}(1-x^{2})f^{\prime}(x)=0; \label{Limit 3}%
\end{equation}
a similar argument establishes the limit as $x\rightarrow-1.$ Let
$\varepsilon>0.$ From\ $($\ref{Crucial Limit 2}$),$ there exists $x^{\ast}%
\in(0,1)$ such that%
\[
\left\vert (1-x)^{2}f^{\prime\prime}(x)\right\vert <\frac{\varepsilon}%
{2}\text{ for }x\in\text{ }[x^{\ast},1).
\]
Integrating this inequality over $[x^{\ast},x]\subset\lbrack x^{\ast},1)$
yields%
\[
\frac{\varepsilon}{2(1-x^{\ast})}+f^{\prime}(x^{\ast})-\frac{\varepsilon
}{2(1-x)}<f^{\prime}(x)<\frac{\varepsilon}{2(1-x)}+f^{\prime}(x^{\ast}%
)-\frac{\varepsilon}{2(1-x^{\ast})}\text{ for }x\in\text{ }[x^{\ast},1).
\]
Multiplying this inequality by $(1-x^{2})$ yields%
\begin{equation}
(1-x^{2})\left(  f^{\prime}(x^{\ast})+\frac{\varepsilon}{2(1-x^{\ast}%
)}\right)  -\frac{\varepsilon(1+x)}{2}<(1-x^{2})f^{\prime}(x)<\frac
{\varepsilon(1+x)}{2}+(1-x^{2})\left(  f^{\prime}(x^{\ast})-\frac{\varepsilon
}{2(1-x^{\ast})}\right)  . \label{Important Inequality}%
\end{equation}
Letting $x\rightarrow1,$ we obtain%
\[
-\varepsilon\leq\lim_{x\rightarrow1}(1-x^{2})f^{\prime}(x)\leq\varepsilon
\]
and this establishes $($\ref{Limit 3}$).$ This completes the proof that
$\mathcal{D}(S)\subset\mathcal{D}(A^{2}).\medskip$\newline\underline
{$\mathcal{D}(A^{2})\subset\mathcal{D}(S)$}$:$\newline Let $f\in
\mathcal{D}(A^{2}).$ Then $f\in\mathcal{D}(A)$ so
\begin{equation}
f,f^{\prime}\in AC_{\mathrm{loc}}(-1,1) \label{A^2=S-5}%
\end{equation}
and%
\begin{equation}
f\in L^{2}(-1,1). \label{A^2=S-6}%
\end{equation}
Moreover, since $\ell\lbrack f]\in\mathcal{D}(A)$, it follows that
\begin{equation}
\ell^{2}[f]=\ell\lbrack\ell\lbrack f]]\in L^{2}(-1,1), \label{A^2=S-7}%
\end{equation}%
\begin{equation}
\ell\lbrack f]=-(1-x^{2})f^{\prime\prime}+2xf^{\prime}\in AC_{\mathrm{loc}%
}(-1,1) \label{A^2=S-8}%
\end{equation}
and%
\begin{equation}
\ell^{\prime}[f]=-(1-x^{2})f^{\prime\prime\prime}+4xf^{\prime}+2f^{\prime}\in
AC_{\mathrm{loc}}(-1,1). \label{A^2=S-9}%
\end{equation}
It is clear that if $f,g\in AC_{\mathrm{loc}}(-1,1)$ then

\begin{enumerate}
\item[(a)$^{\prime}$] $f+g\in AC_{\mathrm{loc}}(-1,1);$

\item[(b)$^{\prime}$] $fg\in AC_{\mathrm{loc}}(-1,1);$

\item[(c)$^{\prime}$] If $g>0$ on $(-1,1)$ then $f/g\in AC_{\mathrm{loc}%
}(-1,1).$
\end{enumerate}

In particular, from $($\ref{A^2=S-5}$)$ and $(\mathrm{b})^{\prime},$ \noindent
we see that $2xf^{\prime}\in AC_{\mathrm{loc}}(-1,1).$ Combining this with
(\textrm{a})$^{\prime}$ and $($\ref{A^2=S-8}$),$ we obtain $(1-x^{2}%
)f^{\prime\prime}\in AC_{\mathrm{loc}}(-1,1).$ Since $1-x^{2}>0$ on $(-1,1)$
we infer from $(\mathrm{c})^{\prime}$ that
\begin{equation}
f^{\prime\prime}\in AC_{\mathrm{loc}}(-1,1). \label{A^2=S-10}%
\end{equation}
Continuing, $-4xf^{\prime\prime}-2f^{\prime}\in AC_{\mathrm{loc}}(-1,1)$ so
from $(\mathrm{a})^{\prime}$ and $($\ref{A^2=S-9}$),$ we have $(1-x^{2}%
)f^{\prime\prime\prime}\in AC_{\mathrm{loc}}(-1,1)$ and it then follows that%
\begin{equation}
f^{\prime\prime\prime}\in AC_{\mathrm{loc}}(-1,1). \label{A^2=S-11}%
\end{equation}
By definition of $\mathcal{D}(A)$ and the fact that $\ell\lbrack
f]\in\mathcal{D}(A),$ we see that
\[
\lim_{x\rightarrow\pm1}(1-x^{2})\ell^{\prime}[f](x)=0;
\]
consequently, in view of $($\ref{A^2=S-1}$),$ we see that
\begin{equation}
0=\lim_{x\rightarrow\pm1}[f,1]_{2}(x)=\lim_{x\rightarrow\pm1}\left(  \left(
(1-x^{2})^{2}f^{\prime\prime}(x)\right)  ^{\prime}-2(1-x^{2})f^{\prime
}(x)\right)  . \label{A^2=S-12}%
\end{equation}
Furthermore since $f\in\mathcal{D}(A),$ we have
\begin{equation}
\lim_{x\rightarrow\pm1}(1-x^{2})f^{\prime}(x)=0 \label{A^2=S-13}%
\end{equation}
so, from $($\ref{A^2=S-12}$),$ we see that%
\begin{equation}
\lim_{x\rightarrow\pm1}\left(  (1-x^{2})^{2}f^{\prime\prime}(x)\right)
^{\prime\prime}=0. \label{A^2=S-14}%
\end{equation}
To finish the proof, we need to show that%
\begin{align}
0  &  =[f,x]_{2}(\pm1)=\lim_{x\rightarrow\pm1}\left(  x[f,1]_{2}%
(x)-(1-x^{2})^{2}f^{\prime\prime}(x)+2(1-x^{2})f(x)\right) \label{A^2=S-15}\\
&  =\lim_{x\rightarrow\pm1}\left(  -(1-x^{2})^{2}f^{\prime\prime}%
(x)+2(1-x^{2})f(x)\right)  \text{ by }(\text{\ref{A^2=S-12}}).\nonumber
\end{align}
We note again, from Green's formula $($%
\ref{Green's Formula for the Legendre Square}$),$ that the limits in
$($\ref{A^2=S-15}$)$ exist and are finite. Since $f\in\mathcal{D}(A),$ we see
from Theorem \ref{ELM}, part $(\mathrm{v})$ that $f\in AC[-1,1]$ and hence%
\begin{equation}
\lim_{x\rightarrow\pm1}(1-x^{2})f(x)=0. \label{A^2=S-16}%
\end{equation}
Thus, proving $($\ref{A^2=S-15}$)$ reduces to showing%
\begin{equation}
\lim_{x\rightarrow\pm1}(1-x^{2})^{2}f^{\prime\prime}(x)=0. \label{A^2=S-17}%
\end{equation}
We show that
\begin{equation}
\lim_{x\rightarrow1}(1-x^{2})^{2}f^{\prime\prime}(x)=0; \label{A^2=S-18}%
\end{equation}
a similar argument will show
\[
\lim_{x\rightarrow-1}(1-x^{2})^{2}f^{\prime\prime}(x)=0.
\]
Suppose, to the contrary, that%
\[
\lim_{x\rightarrow1}(1-x^{2})^{2}f^{\prime\prime}(x)=c\neq0;
\]
without loss of any generality, we can suppose that $c>0.$ Then there exists
$x^{\ast}\in(0,1)$ such that%
\[
(1-x^{2})^{2}f^{\prime\prime}(x)\geq r:=\dfrac{c}{2}\text{ on }[x^{\ast},1);
\]
that is,
\[
f^{\prime\prime}(x)\geq\dfrac{R}{(1-x)^{2}}\text{ on }[x^{\ast},1)
\]
for some $R>0.$ Integrating this inequality over $[x^{\ast},x]\subset\lbrack
x^{\ast},1)$ yields
\begin{align*}
f^{\prime}(x)  &  \geq R\int_{x^{\ast}}^{x}\dfrac{dt}{(1-t)^{2}}+f^{\prime
}(x^{\ast})\\
&  =\dfrac{R}{1-x}+f^{\prime}(x^{\ast})-\frac{R}{1-x^{\ast}}.
\end{align*}
Consequently,
\begin{align*}
(1-x^{2})f^{\prime}(x)  &  \geq R(1+x)+(1-x^{2})\left(  f^{\prime}(x^{\ast
})-\frac{R}{1-x^{\ast}}\right) \\
&  \rightarrow2R>0\quad(\text{as }x\rightarrow1)
\end{align*}
contradicting $($\ref{A^2=S-13}$).$ It follows that $($\ref{A^2=S-18}$)$ holds
and this proves $($\ref{A^2=S-15}$).$ Combining $($\ref{A^2=S-5}$),$
$($\ref{A^2=S-6}$),$ $($\ref{A^2=S-7}$),$ $($\ref{A^2=S-10}$),$ $($%
\ref{A^2=S-11}$),$ $($\ref{A^2=S-12}$)$ and $($\ref{A^2=S-15}$),$ we see that
$f\in\mathcal{D}(A^{2})$ implies $f\in\mathcal{D}(S)$. This completes the
proof of the theorem.
\end{proof}

\section{Proof of Theorem \ref{Main Theorem 2}\label{Proof of Main Theorem 2}}

In order to prove Theorem \ref{Main Theorem 2}, we first need to establish
three preliminary facts, the first of which is the following result.

\begin{lemma}
If $f\in\mathcal{D}(S),$ then
\begin{equation}
\dfrac{1}{1-x^{2}}\left(  (1-x^{2})^{2}f^{\prime\prime}(x)\right)  ^{\prime
}\in L^{2}(-1,1). \label{L^2 condition}%
\end{equation}

\end{lemma}

\begin{proof}
Let $f\in\mathcal{D}(S)=\mathcal{D}(A^{2})$ so\ $f^{\prime}\in L^{2}(-1,1)$,
$[f,1]_{2}(\pm1)=0$ and $\ell^{2}[f]\in L^{2}(-1,1).$ We apply the CE Theorem
on $[0,1)$ with $\psi(x)=1$, $\varphi(x)=-1/(1-x^{2})$ and $w(x)=1. $ These
functions satisfy the conditions of this theorem on $[0,1)$ so%
\[
\frac{-1}{1-x^{2}}\int_{x}^{1}\ell^{2}[f](t)dt\in L^{2}(0,1).
\]
However, using (\ref{[f,1]_2}), a calculation shows
\begin{align*}
\dfrac{-1}{1-x^{2}}\int_{x}^{1}\ell^{2}[f](t)dt  &  =\dfrac{-1}{1-x^{2}}%
\int_{x}^{1}\left[  \left(  \left(  1-t^{2}\right)  ^{2}f^{\prime\prime
}(t)\right)  ^{\prime\prime}-2\left(  (1-t^{2})f^{\prime}(t)\right)  ^{\prime
}\right]  dt\\
&  =\dfrac{-1}{1-x^{2}}\left[  \lim_{x\rightarrow1}\left(  \left(  \left(
1-x^{2}\right)  ^{2}f^{\prime\prime}(x)\right)  ^{\prime}-2(1-x^{2})f^{\prime
}(x)\right)  \right] \\
&  \qquad\qquad+\frac{1}{1-x^{2}}\left[  \left(  (1-x^{2})^{2}f^{\prime\prime
}(x)\right)  ^{\prime}-2(1-x^{2})f^{\prime}(x)\right] \\
&  =\dfrac{-1}{1-x^{2}}\left[  \lim_{x\rightarrow1}\text{ }[f,1]_{2}%
(x)-\left(  (1-x^{2})^{2}f^{\prime\prime}(x)\right)  ^{\prime}+2(1-x^{2}%
)f^{\prime}(x)\right] \\
&  =\dfrac{1}{1-x^{2}}\left(  (1-x^{2})^{2}f^{\prime\prime}(x)\right)
^{\prime}-2f^{\prime}(x)\text{.}%
\end{align*}
A similar calculation shows that
\[
\dfrac{1}{1-x^{2}}\left(  (1-x^{2})^{2}f^{\prime\prime}(x)\right)  ^{\prime
}-2f^{\prime}(x)\in L^{2}(-1,0]
\]
and hence%
\[
\dfrac{1}{1-x^{2}}\left(  (1-x^{2})^{2}f^{\prime\prime}(x)\right)  ^{\prime
}-2f^{\prime}(x)\in L^{2}(-1,1).
\]
Since $f^{\prime}\in L^{2}(-1,1),$ we see, by linearity, that
\[
\dfrac{1}{1-x^{2}}\left(  (1-x^{2})^{2}f^{\prime\prime}(x)\right)  ^{\prime
}\in L^{2}(-1,1).
\]

\end{proof}

\begin{lemma}
\label{A_2 limit}For $f\in\mathcal{D}(S),$ we have
\begin{equation}
\lim_{x\rightarrow\pm1}(1-x^{2})^{2}f^{\prime\prime}(x)=0.
\label{A_2 limit = 0}%
\end{equation}

\end{lemma}

\begin{proof}
Let $f\in\mathcal{D}(S)=\mathcal{D}(A^{2}).$ Since $f\in\mathcal{D}(A),$ we
have $f\in AC[-1,1]$ so%
\begin{equation}
\lim_{x\rightarrow\pm1}(1-x^{2})f(x)=0. \label{Simple Limit}%
\end{equation}
Furthermore, we have
\begin{equation}
0=\lim_{x\rightarrow\pm1}[f,1]_{2}(x)=\lim_{x\rightarrow\pm1}\left(  \left(
(1-x^{2})^{2}f^{\prime\prime}(x)\right)  ^{\prime}-2(1-x^{2})f^{\prime
}(x)\right)  . \label{D(S) limit}%
\end{equation}
Consequently, from $($\ref{Simple Limit}$)$ and $($\ref{D(S) limit}$),$ we
find that
\[
0=\lim_{x\rightarrow\pm1}[f,x]_{2}(x)=\lim_{x\rightarrow\pm1}\left(
x[f,1]_{2}(x)-(1-x^{2})^{2}f^{\prime\prime}(x)+2(1-x^{2})f(x)\right)
=-\lim_{x\rightarrow\pm1}(1-x^{2})^{2}f^{\prime\prime}(x).
\]

\end{proof}

The last preliminary result is the following theorem. Since $\mathcal{D}%
(S)=\mathcal{D}(A^{2}),$ this next result generalizes the well-known result
for $\mathcal{D}(A)$ established in Theorem \ref{ELM}, part (v).

\begin{theorem}
\label{f'' in L^2}If $f\in\mathcal{D}(S),$ then%
\[
f^{\prime\prime}\in L^{2}(-1,1).
\]
Moreover,
\begin{equation}
pf^{\prime\prime}\in L^{2}(-1,1)\label{pf'' is in L^2(-1,1)}%
\end{equation}
for any bounded, Lebesgue measurable function $p$, including any polynomial$.$
\end{theorem}

\begin{proof}
Once we establish $f^{\prime\prime}\in L^{2}(-1,1),$ the statement in
$($\ref{pf'' is in L^2(-1,1)}$),$ for any bounded measurable function, follows
clearly. Let $f\in\mathcal{D}(S).$ We prove that $f^{\prime\prime}\in
L^{2}(0,1);$ a similar proof will establish $f^{\prime\prime}\in L^{2}(-1,0)$
and prove the theorem. We again use the CE Theorem with $\psi(x)=1-x^{2}$,
$\varphi(x)=1/(1-x^{2})^{2}$ and $w(x)=1$ on $[0,1).$ Indeed, from the CE
Theorem and $($\ref{L^2 condition}$)$, we find that%
\[
\dfrac{-1}{(1-x^{2})^{2}}\int_{x}^{1}(1-t^{2})\left(  \dfrac{1}{1-t^{2}%
}\left(  (1-t^{2})^{2}f^{\prime\prime}(t)\right)  ^{\prime}\right)  dt\in
L^{2}(0,1).
\]
However, from Lemma \ref{A_2 limit},
\begin{align*}
&  \dfrac{-1}{(1-x^{2})^{2}}\int_{x}^{1}(1-t^{2})\left(  \dfrac{1}{1-t^{2}%
}\left(  (1-t^{2})^{2}f^{\prime\prime}(t)\right)  ^{\prime}\right)  dt\\
&  =\frac{-1}{(1-x^{2})^{2}}\left(  \lim_{x\rightarrow1}(1-x^{2})^{2}%
f^{\prime\prime}(x)-(1-x^{2})^{2}f^{\prime\prime}(x)\right) \\
&  =f^{\prime\prime}(x).
\end{align*}

\end{proof}

We are now in position to prove Theorem \ref{Main Theorem 2}, specifically
$B=\mathcal{D}(S),$ where $B$ is defined in (\ref{LD Representation}) and
$\mathcal{D}(S)$ is given in (\ref{D(S)}).

\begin{proof}
\underline{$B\subset\mathcal{D}(S)$}:\medskip

\noindent Let $f\in B.$ We assume that $f$ is real-valued on $(-1,1).$ We
begin by showing, using the CE Theorem, that the condition
\[
(1-x^{2})^{2}f^{(4)}\in L^{2}(-1,1)
\]
implies the two conditions
\begin{equation}
(1-x^{2})f^{\prime\prime\prime}\in L^{2}(-1,1) \label{f''' condition}%
\end{equation}
and%
\begin{equation}
f^{\prime\prime}\in L^{2}(-1,1). \label{f'' condition}%
\end{equation}
Regarding $($\ref{f''' condition}$),$ we will show
\begin{equation}
(1-x^{2})f^{\prime\prime\prime}\in L^{2}(0,1); \label{f''' condition on (0,1)}%
\end{equation}
a similar proof will yield
\begin{equation}
(1-x^{2})f^{\prime\prime\prime}\in L^{2}(-1,0)
\label{f''' condition on (-1,0)}%
\end{equation}
and, together, they establish $($\ref{f''' condition}$).$ Since $(1-x^{2}%
)^{2}f^{(4)}\in L^{2}(0,1),$ we use the CE Theorem on $[0,1)$ with%
\[
\varphi(x)=(1-x^{2})^{-2},\text{ }\psi(x)=1-x^{2}\text{ and }w(x)=1\quad
(x\in\lbrack0,1)).
\]
It follows that
\begin{align*}
(1-x^{2})f^{\prime\prime\prime}(x)  &  =(1-x^{2})\int_{0}^{x}\dfrac
{1}{(1-t^{2})^{2}}(1-t^{2})^{2}f^{(4)}(t)dt+f^{\prime\prime\prime}%
(0)(1-x^{2})\\
&  \in L^{2}(0,1).
\end{align*}
To see $($\ref{f'' condition}$),$ we apply the CE Theorem once again on
$[0,1)$ to prove that
\[
f^{\prime\prime}\in L^{2}(0,1);
\]
a similar argument will show that $f^{\prime\prime}\in L^{2}(-1,0).$ To this
end, let
\[
\varphi(x)=(1-x^{2})^{-1},\text{ }\psi(x)=1\text{ and }w(x)=1\quad(x\in
\lbrack0,1)).
\]
In this case, we see that%
\[
f^{\prime\prime}(x)=\int_{0}^{x}\dfrac{1}{1-t^{2}}\left(  (1-t^{2}%
)f^{\prime\prime\prime}(t)\right)  dt+f^{\prime\prime}(0)\in L^{2}(0,1).
\]
Consequently, we see that
\[
f,f^{\prime}\in AC[-1,1]\subset L^{2}(-1,1).
\]
Moreover, it is clear that $g(x)(1-x^{2})f^{\prime\prime\prime}(x),$
$g(x)f^{\prime\prime}(x)$ and $g(x)f^{\prime}(x)$ all belong to $L^{2}(-1,1)$
for any bounded, measurable function $g$ on $(-1,1).$ Hence%
\begin{align*}
\ell^{2}[f](x)  &  =(1-x^{2})^{2}f^{(4)}(x)-8x(1-x^{2})f^{\prime\prime\prime
}(x)+(14x^{2}-6)f^{\prime\prime}(x)+4xf^{\prime}(x)\\
&  \in L^{2}(-1,1)
\end{align*}
and, in particular,
\begin{equation}
4xf^{\prime}\in L^{2}(-1,1). \label{2xf' is in L^(-1,1)}%
\end{equation}
It remains to show that
\begin{equation}
\lim_{x\rightarrow\pm1}[f,1]_{2}(x)=\lim_{x\rightarrow\pm1}[f,x]_{2}(x)=0.
\label{Two limits}%
\end{equation}
Since $1,x\in$ $\Delta_{2,\mathrm{\max}},$ we see from Green's formula in
$($\ref{Green's Formula for the Legendre Square}$)$ that the limits in
$($\ref{Two limits}$)$ both exist and are finite. Now $f^{\prime}\in
AC[-1,1]\ $so%
\[
\lim_{x\rightarrow\pm1}(1-x^{2})f^{\prime}(x)=0.
\]
Consequently,
\begin{align*}
\lim_{x\rightarrow\pm1}[f,1]_{2}(x)  &  =\lim_{x\rightarrow\pm1}\left(
((1-x^{2})^{2}f^{\prime\prime}(x))^{\prime}-2(1-x^{2})f^{\prime}(x)\right) \\
&  =\lim_{x\rightarrow\pm1}((1-x^{2})^{2}f^{\prime\prime}(x))^{\prime}.
\end{align*}
We claim that
\begin{equation}
\lim_{x\rightarrow1}((1-x^{2})^{2}f^{\prime\prime}(x))^{\prime}=0; \label{BC1}%
\end{equation}
a similar proof will establish%
\[
\lim_{x\rightarrow-1}((1-x^{2})^{2}f^{\prime\prime}(x))^{\prime}=0.
\]
Suppose to the contrary that%
\[
\lim_{x\rightarrow1}((1-x^{2})^{2}f^{\prime\prime}(x))^{\prime}=c\neq0;
\]
we can assume that $c>0.$ It follows that there exists $x^{\ast}\in(0,1)$ such
that%
\begin{equation}
((1-x^{2})^{2}f^{\prime\prime}(x))^{\prime}\geq r:=\frac{c}{2}>0\quad
(x\in\lbrack x^{\ast},1)). \label{B_2 contained in D(S) inequality}%
\end{equation}
Note that since
\[
((1-x^{2})^{2}f^{\prime\prime}(x))^{\prime}=(1-x^{2})^{2}f^{\prime\prime
\prime}(x)-4x(1-x^{2})f^{\prime\prime}(x),
\]
we see that the inequality in $($\ref{B_2 contained in D(S) inequality}$)$ can
be rewritten as%
\begin{equation}
(1-x^{2})f^{\prime\prime\prime}(x)-4xf^{\prime\prime}(x)\geq\frac{r}{1-x^{2}%
}\text{ on }[x^{\ast},1). \label{Key inequality}%
\end{equation}
However, from (\ref{pf'' is in L^2(-1,1)}) and (\ref{f''' condition}), we see
that%
\[
(1-x^{2})f^{\prime\prime\prime}(x)-4xf^{\prime\prime}(x)\in L^{2}(-1,1)
\]
so the inequality in (\ref{Key inequality}) is not possible. Hence
$($\ref{BC1}$)$ is established and thus%
\[
\lim_{x\rightarrow\pm1}[f,1]_{2}(x)=0.
\]
We now show that%
\[
\lim_{x\rightarrow\pm1}[f,x]_{2}(x)=0.
\]
Since the argument for $x\rightarrow-1$ mirrors the proof for $x\rightarrow1,$
we will only show that
\begin{equation}
\lim_{x\rightarrow1}[f,x]_{2}(x)=0. \label{BC2}%
\end{equation}
Now, since $f\in AC[-1,1],$ we see that $\lim_{x\rightarrow1}(1-x^{2})f(x)=0;$
moreover, using $($\ref{BC1}$),$
\[
\lim_{x\rightarrow1}[f,x]_{2}(x)=\lim_{x\rightarrow1}\left(  x[f,1]_{2}%
(x)-(1-x^{2})^{2}f^{\prime\prime}(x)+2(1-x^{2})f(x)\right)  =-\lim
_{x\rightarrow1}(1-x^{2})^{2}f^{\prime\prime}(x).
\]
Suppose that%
\[
\lim_{x\rightarrow1}(1-x^{2})^{2}f^{\prime\prime}(x)=d\neq0;
\]
we can assume that $d>0.$ Then, with possibly different $x^{\ast}$ as given in
the above argument, there exists a $x^{\ast}\in(0,1)$ with%
\[
(1-x^{2})^{2}f^{\prime\prime}(x)\geq d^{\prime}:=\frac{d}{2}\quad(x\in\lbrack
x^{\ast},1)).
\]
Hence%
\[
f^{\prime\prime}(x)\geq\frac{d^{\prime}}{(1-x^{2})^{2}}\quad(x\in\lbrack
x^{\ast},1)).
\]
However, this implies that $f^{\prime\prime}\notin L^{2}(0,1),$ contradicting
$($\ref{f'' condition}$)$. Thus $($\ref{BC2}$)$ is established and this
completes the proof that $B\subset\mathcal{D}(S).\medskip$

\noindent\underline{$\mathcal{D}(S)\subset B$}:\medskip\newline\noindent Let
$f\in\mathcal{D}(S)$. We need only to show that
\begin{equation}
(1-x^{2})^{2}f^{(4)}\in L^{2}(-1,1). \label{Theorem condition}%
\end{equation}
Since, by Theorem \ref{f'' in L^2}, $f^{\prime\prime}\in L^{2}(-1,1),$ we see
that $gf^{\prime\prime}\in L^{2}(-1,1)$ for any bounded, measurable function
$g$ on $(-1,1).$ In particular, it is the case that%
\begin{equation}
4xf^{\prime\prime}\in L^{2}(-1,1) \label{L^2 condition 1}%
\end{equation}
and%
\begin{equation}
(14x^{2}-6)f^{\prime\prime}\in L^{2}(-1,1). \label{L^2 condition again}%
\end{equation}
By $($\ref{L^2 condition}$)$,
\begin{equation}
(1-x^{2})f^{\prime\prime\prime}(x)-4xf^{\prime\prime}(x)=\dfrac{1}{1-x^{2}%
}\left(  (1-x^{2})^{2}f^{\prime\prime}(x)\right)  ^{\prime}\in L^{2}(-1,1).
\label{Another L^2 condition}%
\end{equation}
By linearity, it follows from $($\ref{L^2 condition 1}$)$ and $($%
\ref{Another L^2 condition}$)$ that%
\[
(1-x^{2})f^{\prime\prime\prime}\in L^{2}(-1,1).
\]
Consequently, $g(1-x^{2})f^{\prime\prime\prime}\in L^{2}(-1,1)$ for every
bounded, measurable function $g$ on $(-1,1);$ in particular,
\begin{equation}
8x(1-x^{2})f^{\prime\prime\prime}\in L^{2}(-1,1).
\label{Yet another L^2 condition}%
\end{equation}
Furthermore, since $f^{\prime}\in L^{2}(-1,1),$ it follows that
\begin{equation}
4xf^{\prime}(x)\in L^{2}(-1,1). \label{Yet yet another L^2 condition}%
\end{equation}
Finally, since $\ell^{2}[f]\in L^{2}(-1,1),$ we see from $($%
\ref{Square of Legendre DE}$),$ $($\ref{L^2 condition again}$)$,
$($\ref{Yet another L^2 condition}$)$ and $($%
\ref{Yet yet another L^2 condition}$)$ that%
\begin{align*}
(1-x^{2})^{2}f^{(4)}  &  =\ell^{2}[f]+8x(1-x^{2})f^{\prime\prime\prime
}-(14x^{2}-6)f^{\prime\prime}-4xf^{\prime}\\
&  \in L^{2}(-1,1).
\end{align*}
This establishes $($\ref{Theorem condition}$)$ and proves $\mathcal{D}%
(S)\subset B.$ This completes the proof of Theorem \ref{Main Theorem 2}.
\end{proof}

\section{Proof of Theorem \ref{Main Theorem 3}\label{Proof of Main Theorem 3}}

We now prove Theorem \ref{Main Theorem 3}, namely $\mathcal{D}(S)=D,$ where
$\mathcal{D}(S)$ is given in (\ref{D(S)}) and $D$ is defined in
(\ref{New Domain D}).

\begin{proof}
Since functions $f$ in both $\mathcal{D}(S)$ and $D$ satisfy the `maximal
domain' conditions $f^{(j)}\in AC_{\mathrm{loc}}(-1,1)$ $(j=0,1,2,3),$ $f\in
L^{2}(-1,1)$ and $\ell^{2}[f]\in L^{2}(-1,1),$ we need only to prove that the
other properties in their definitions hold.

\noindent\underline{$\mathcal{D}(S)\subset D$}$:\medskip$\newline Let
$f\in\mathcal{D}(S)=\mathcal{D}(A^{2}).$ Then $f\in\mathcal{D}(A)$ so
\begin{equation}
\lim_{x\rightarrow\pm1}(1-x^{2})f^{\prime}(x)=0. \label{NC-1}%
\end{equation}
Moreover,
\begin{align}
0  &  =[f,1]_{2}(\pm1)\nonumber\\
&  =\lim_{x\rightarrow\pm1}\left(  \left(  (1-x^{2})^{2}f^{\prime\prime
}(x)\right)  ^{\prime}-2(1-x^{2})f^{\prime}(x)\right) \nonumber\\
&  =\lim_{x\rightarrow\pm1}\left(  (1-x^{2})^{2}f^{\prime\prime}(x)\right)
^{\prime}. \label{NC-2}%
\end{align}
The identities in $($\ref{NC-1}$)$ and $($\ref{NC-2}$)$ prove that
$\mathcal{D}(S)\subset D.$\newline\noindent\underline{$D\subset\mathcal{D}%
(S)$}$:\medskip$\newline Let $f\in D.$ Clearly,
\begin{align}
\lbrack f,1]_{2}(\pm1)  &  =\lim_{x\rightarrow\pm1}\left(  \left(
(1-x^{2})^{2}f^{\prime\prime}(x)\right)  ^{\prime}-2(1-x^{2})f^{\prime
}(x)\right) \label{NC-3}\\
&  =0\nonumber
\end{align}
so we need to show that
\begin{equation}
\lim_{x\rightarrow\pm1}[f,x]_{2}(\pm1)=0. \label{NC-4}%
\end{equation}
We remark that the limits in $($\ref{NC-4}$)$ exist (by Green's formula) and
are finite. \newline\underline{Claim:} $\ell^{\prime}[f]\in L^{2}%
(-1,1).$\newline To see this, recall the two representations of $\ell
^{2}[\cdot]$: the one given in $($\ref{Square of Legendre DE}$)$ and the one
given in $($\ref{ell^2 identity}$).$ Since $\ell^{2}[f]\in L^{2}(-1,1)$, we
apply the CE Theorem on $[0,1)$ with $\varphi(x)=(1-x^{2})^{-1},$ $\psi(x)=1$
and $w(x)=1$ to obtain%
\[
\frac{1}{1-x^{2}}\int_{x}^{1}\ell^{2}[f](t)dt\in L^{2}(0,1).
\]
However, from $($\ref{Square of Legendre DE}$)$ and $($\ref{ell^2 identity}%
$),$ we see that
\begin{align*}
&  \frac{1}{1-x^{2}}\int_{x}^{1}\ell^{2}[f](t)dt\\
&  =\frac{1}{1-x^{2}}\left(  \lim_{x\rightarrow1}\left(  \left(  (1-x^{2}%
)^{2}f^{\prime\prime}(x)\right)  ^{\prime}-2(1-x^{2})f^{\prime}(x)\right)
+(1-x^{2})\ell^{\prime}[f](x)\right) \\
&  =\ell^{\prime}[f](x)\text{ by }(\text{\ref{NC-3}});
\end{align*}
a similar calculation shows that $\ell^{\prime}[f]\in L^{2}(-1,0).$ It follows
that $\ell\lbrack f]\in AC[-1,1]\subset L^{2}(-1,1).$ We again apply the CE
Theorem on $[0,1)$ with $\varphi(x)=(1-x^{2})^{-1},$ $\psi(x)=1$ and $w(x)=1$
to obtain%
\[
\frac{1}{1-x^{2}}\int_{x}^{1}\ell\lbrack f](t)dt\in L^{2}(0,1).
\]
Another calculation shows that
\begin{align*}
&  \frac{1}{1-x^{2}}\int_{x}^{1}\ell\lbrack f](t)dt=\frac{-1}{1-x^{2}}\int
_{x}^{1}\left(  (1-t^{2})f^{\prime}(t)\right)  ^{\prime}dt\\
&  =\frac{-1}{1-x^{2}}\left(  \lim_{x\rightarrow1}(1-x^{2})f^{\prime
}(x)-(1-x^{2})f^{\prime}(x)\right) \\
&  =f^{\prime}(x)\text{ by definition of }D;
\end{align*}
a similar argument shows that $f^{\prime}\in L^{2}(-1,0)$. Hence%
\begin{equation}
f^{\prime}\in L^{2}(-1,1). \label{NC-6}%
\end{equation}
Thus, $f\in AC[-1,1]$ and
\begin{equation}
\lim_{x\rightarrow\pm1}(1-x^{2})f(x)=0. \label{NC-7}%
\end{equation}
From $($\ref{NC-3}$)$ and $($\ref{NC-7}$),$ we see that%
\begin{align*}
\lim_{x\rightarrow\pm1}[f,x]_{2}(x)  &  =\lim_{x\rightarrow\pm1}\left(
x[f,1](x)-(1-x^{2})^{2}f^{\prime\prime}(x)+2(1-x^{2})f(x)\right) \\
&  =-\lim_{x\rightarrow\pm1}(1-x^{2})^{2}f^{\prime\prime}(x).
\end{align*}
To establish $($\ref{NC-4}$),$ it now suffices to prove that%
\begin{equation}
\lim_{x\rightarrow\pm1}(1-x^{2})^{2}f^{\prime\prime}(x)=0. \label{NC-8}%
\end{equation}
Since the proof as $x\rightarrow-1$ is similar to the proof that
$x\rightarrow1,$ we will only show that
\[
\lim_{x\rightarrow1}(1-x^{2})^{2}f^{\prime\prime}(x)=0;
\]
By way of contradiction, suppose that
\[
\lim_{x\rightarrow1}(1-x^{2})^{2}f^{\prime\prime}(x)=c\neq0;
\]
without loss of generality, we may assume that $c>0.$ Then there exists
$x^{\ast}\in(0,1)$ such that%
\[
f^{\prime\prime}(x)\geq\frac{c}{2(1-x^{2})^{2}}\geq\frac{c}{8(1-x)^{2}}%
\quad(x\in\lbrack x^{\ast},1)).
\]
Integrating this inequality over $[x^{\ast},x]\subset\lbrack x^{\ast},1)$
yields%
\[
f^{\prime}(x)\geq\frac{c}{8(1-x)}+f^{\prime}(x^{\ast})-\dfrac{c}{8(1-x^{\ast
})}\quad(x\in\lbrack x^{\ast},1)).\quad
\]
But this contradicts $($\ref{NC-6}$).$ It follows that $($\ref{NC-8}$)$ holds
and this, in turn, establishes $($\ref{NC-4}$).$ Consequently, $D\subset
\mathcal{D}(S)$ and this completes the proof of the theorem.
\end{proof}

As revealed in the proofs of Theorems \ref{Main Theorem 1},
\ref{Main Theorem 2}, \ref{Main Theorem 3} and \ref{f'' in L^2}, we have the
following interesting result.

\begin{corollary}
\label{Smoothness of D(A^2)}If $f\in\mathcal{D}(A^{2})=\mathcal{D}(S)=B=D,$ then
\end{corollary}

\begin{enumerate}
\item[(i)] $f^{\prime\prime}\in L^{2}(-1,1)$ so $f,f^{\prime}\in AC[-1,1]; $

\item[(ii)] $\ell^{\prime}[f]\in L^{2}(-1,1)$ and $\ell\lbrack f]\in
AC[-1,1].$
\end{enumerate}

\begin{remark}
As discussed in Section \ref{GKN Operator}, the minimal operator
$T_{2,\mathrm{\min}}$ in $L^{2}(-1,1)$ generated by $\ell^{2}[\cdot]$ has
deficiency index $(4,4).$ From the GKN Theorem $($see \cite[Theorem 4, Section
18.1]{Naimark}$),$ GKN boundary conditions for any self-adjoint extension of
$T_{2,\mathrm{\min}}$ in $L^{2}(-1,1)$ are restrictions of the maximal domain
$\Delta_{2,\mathrm{\max}}$ and have the appearance $($see $( $%
\ref{Form of GKN BCs}$))$
\[
\lbrack f,f_{j}]_{2}(1)-[f,f_{j}]_{2}(-1)=0\quad(f\in\Delta_{2,\mathrm{\max}%
},j=1,2,3,4),
\]
where $\{f_{j}\}_{j=1}^{4}\subset\Delta_{2,\mathrm{\max}}$ are linearly
independent modulo the minimal domain $\Delta_{2,\mathrm{\min}}.$ Taking into
account $[\cdot,\cdot]_{2}$, defined in $($\ref{Sesquilinear Form for Square}%
$),$ it is clear that the boundary conditions given in $($\ref{New Domain D}%
$)$ are not GKN boundary conditions.
\end{remark}

\section{Concluding Remarks\label{Concluding Remarks}}

In \cite{Everitt-Littlejohn-Wellman}, the authors showed that, for
$n\in\mathbb{N},$ the $n^{th}$ composite power of the Legendre differential
expression $\ell\lbrack\cdot]$ is explicitly given by%
\begin{equation}
\ell^{n}[y](x)=\sum_{j=1}^{n}(-1)^{j}\left\{
\begin{array}
[c]{c}%
n\\
j
\end{array}
\right\}  _{1}\left(  (1-x^{2})^{j}y^{(j)}(x)\right)  ^{(j)},
\label{nth power of Legendre expression}%
\end{equation}
where the numbers
\[
\left\{
\begin{array}
[c]{c}%
n\\
j
\end{array}
\right\}  _{1}:=\sum_{r=0}^{j}(-1)^{r+j}\frac{(2r+1)\left(  r^{2}+r\right)
^{n}}{(j-r)!(j+r+1)!}%
\]
are the so-called Legendre-Stirling numbers, a subject of current study in
combinatorics (for example, see \cite{AEGL}, \cite{AGL},
\cite{Andrews-Littlejohn}, \cite{Egge} and \cite{GLN}). The expression in
(\ref{nth power of Legendre expression}) is the key in generating the domain
$\mathcal{D}(A^{n})$ of $A^{n}$ given in (\ref{LD Representation General}).

We conjecture:\medskip

\noindent\underline{Conjecture} Let $A$ denote the Legendre polynomials
self-adjoint operator defined in $($\ref{The operator A}$)$ and $($%
\ref{D(A)}$).$ For $n\in\mathbb{N}$, let $\ell^{n}[\cdot]$ be given as in
$($\ref{nth power of Legendre expression}$)$ and let $[\cdot,\cdot]_{n}$ be
the sequilinear form associated with the maximal domain $\Delta
_{n,\mathrm{\max}}$ of $\ell^{n}[\cdot]$ in $L^{2}(-1,1)$. Then $A_{n}%
=B_{n}=C_{n}=D_{n}$, where\medskip

$%
\begin{array}
[c]{l}%
\text{(i) \ }A_{n}:=\mathcal{D}(A^{n})\medskip,\\
\text{(ii) \thinspace}B_{n}:=\{f:(-1,1)\rightarrow\mathbb{C\mid}f,f^{\prime
},\ldots,f^{(2n-1)}\in AC_{\mathrm{loc}}(-1,1);(1-x^{2})^{n}f^{(2n)}\in
L^{2}(-1,1)\},\medskip\\
\text{(iii) }C_{n}:=\{f:(-1,1)\rightarrow\mathbb{C}\mid f,f^{\prime}%
,\ldots,f^{(2n-1)}\in AC_{\mathrm{loc}}(-1,1);f,\ell^{n}[f]\in L^{2}%
(-1,1);\smallskip\\
\qquad\qquad\qquad\qquad\qquad\qquad\lbrack f,x^{j}]_{n}(\pm1)=0\text{ for
}j=0,1,2,\ldots,n-1\},\smallskip\\
\text{(iv) }D_{n}:=\{f:(-1,1)\rightarrow\mathbb{C}\mid f,f^{\prime}%
,\ldots,f^{(2n-1)}\in AC_{\mathrm{loc}}(-1,1);f,\ell^{n}[f]\in L^{2}%
(-1,1);\smallskip\\
\qquad\qquad\qquad\qquad\qquad\qquad\lim_{x\rightarrow\pm1}\left(
(1-x^{2})^{j}y^{(j)}(x)\right)  ^{(j-1)}=0\text{ for }j=1,2,\ldots
,n\}.\medskip
\end{array}
$

By repeated applications of the CE Theorem, it is not difficult to establish
that if $f\in B_{n},$ then $f^{(n)}\in L^{2}(-1,1);$ this result generalizes
Theorem \ref{ELM}, part (iii) ($n=1)$ and Corollary \ref{Smoothness of D(A^2)}%
, part (i) ($n=2)$.

We remark that, in (iii) above, we can replace the monomials $\{x^{j}%
\}_{j=0}^{n-1}$ by the Legendre polynomials $\{P_{j}\}_{j=0}^{n-1}.$ One of
the difficulties in our efforts to try and prove this conjecture lies in the
fact that the corresponding sesquilinear form $[\cdot,\cdot]_{n},$ associated
with the $n^{th}$ power $\ell^{n}[\cdot]$, is unwieldy at the present time.
\medskip

\noindent\underline{Acknowledgement} The author LLL would like to recognize
the many years of mentoring from W. N. $($Norrie$)$ Everitt, who passed away
on July 17, 2011 at the age of 87. The Legendre differential expression and
Legendre polynomials were among the favorites of Norrie's many mathematical
interests. Quinn and I both felt we were guided by the hand of Norrie in the
writing of this paper.

\end{document}